\documentclass{amsart}
\usepackage{graphicx}
\usepackage{amsmath, amsthm, amssymb}
\usepackage{bm,txfonts,color}
\usepackage{txfonts,color}
\usepackage{ytableau}

\newtheorem{thm}{Theorem}[section]
\newtheorem{cor}[thm]{Corollary}
\newtheorem{lem}[thm]{Lemma}
\newtheorem{conj}[thm]{Conjecture}
\newtheorem{rmk}[thm]{Remark}
\newtheorem{definition}[thm]{Definition}
\newtheorem{proposition}[thm]{Proposition}
\newtheorem{example}[thm]{Example}
\newcommand{\kso}{s^{(1)}}
\newcommand{\kst}{s^{(2)}}
\newcommand{\bmt}{\bm{T}}
\newcommand{\bmlambda}{{\bm{\lambda}}}
\newcommand{\bmmu}{{\bm{\mu}}}

\newcommand{\bmeta}{\bm{\eta}}
\DeclareMathOperator{\Inv}{Inv}
\DeclareMathOperator{\id}{ID}

\DeclareMathOperator{\inv}{inv}
\DeclareMathOperator{\des}{Des}
\DeclareMathOperator{\ssty}{SSTY}
\DeclareMathOperator{\sty}{STY}

\title{Linear relations on LLT polynomials and their k-Schur positivity for k=2 }
\author{Seung Jin Lee}
\address{Department of Mathematical Sciences \\ Research institute of Mathematics \\ Seoul National University \\ Gwanak-ro 1, Gwanak-gu \\ Seoul 151-747, Republic of Korea}
\email{lsjin@snu.ac.kr}

\begin{document}
\begin{abstract}
LLT polynomials are $q$-analogues of product of Schur functions that are known to be Schur-positive by Grojnowski and Haiman. However, there is no known combinatorial formula for the coefficients in the Schur expansion. Finding such a formula also provides Schur positivity of Macdonald polynomials.
On the other hand, Haiman and Hugland conjectured that LLT polynomials for skew partitions lying on $k$ adjacent diagonals are $k$-Schur positive, which is much stronger than Schur positivity. In this paper, we prove the conjecture for $k=2$ by analyzing unicellular LLT polynomials. We first present a linearity theorem for unicellular LLT polynomials for $k=2$. By analyzing linear relations between LLT polynomials with known results on LLT polynomials for rectangles, we provide the $2$-Schur positivity of the unicellular LLT polynomials as well as LLT polynomials appearing in Haiman-Hugland conjecture for $k=2$.
\end{abstract}
\maketitle
\section{Introduction}

LLT polynomials are certain family of symmetric functions indexed by $d$-tuple of skew partitions, introduced by Lascoux, Leclerc, and Thibon \cite{LLT} in the study of quantum affine algebras and unipotent varieties. Later Haglund, Haiman and Loehr \cite{HHL05} proved that Macdonald polynomials are positive sums of LLT polynomial indexed by $d$-tuple of ribbons. Grojnowski and Haiman \cite{GH07} proved that LLT polynomials are Schur-positive using Kazhdan-Lusztig theory. However, their proof does not provide a manifestly positive formula so finding combinatorial formulas for expansions of Macdonald polynomials and LLT polynomials remains a wide open problem. The best known result is the formula for $d=3$ due to Blasiak \cite{Bla}. See \cite{Bla} for more history about LLT polynomials.\\

In his 2006 ICM talk, Haiman announced a conjecture made by Haiman and Hugland stating that (conjugate of) LLT polynomials indexed by $d$-tuple of skew partition that lies in $k$-adjacent diagonals are $k$-Schur positive (Theorem \ref{haiman}), which is much stronger than Schur positivity. In this paper, we prove the conjecture for $k=2$. The proof is divided into two steps: we first prove that there are linear relations between (unicellular) LLT polynomials, which is enough to determine all unicellular LLT polynomials with $k=2$ from LLT polynomials indexed by dominos. Then we prove that LLT polynomials indexed by dominos are power of $q$ times a $2$-Schur function by using the fact that LLT polynomials indexed by rectangles are the same as generalized Hall-Littlewood polynomials, proved in \cite{GH07}. The main theorem (Theorem \ref{convthm}) shows that unicellular LLT polynomials with $k=2$ are positive sums of $2$-Schur functions where the exponents of $q$ change linearly as the index set of unicellular LLT polynomials change, providing a very nice formula.\\

It is worth nothing that there is a plethysm relation between unicellular LLT polynomials and chromatic quasisymmetric functions, first found in \cite{CM}. Therefore linear relations between unicellular LLT polynomials also hold for corresponding chromatic quasisymmetric functions. Chromatic quasisymmetric functions for the case $k=2$ we considered are studied by Cho and Huh \cite{CH}, and by Harada and Precup \cite{HP}.\\

The structure of the paper is as follows: In Section 2, we define LLT polynomials and state Haiman-Hugland conjecture precisely. In Section 3 and 4, we present useful linear relations and prove the linearity theorem (Theorem \ref{convthm}). In Section 5 we compare LLT polynomials indexed by dominos and $2$-Schur function. In Section 6 we prove Haiman-Hugland conjecture for $k=2$.

\section{Preliminary}
LLT polynomials are certain $q$-analogs of products of skew Schur functions defined by Lascoux, Leclerc, and Thibon \cite{LLT}. Below we give an alternative definition presented in \cite{HHL05}. Let $\bm{\lambda}= (\lambda^{(0)},\lambda^{(1)}, \ldots, \lambda^{(d-1)})$ be a $d$-tuple of (skew) tableaux. Define
\begin{align*}
\ssty_d({\bm \lambda}) &= \{\text{semi-standard }d\text{-tuples of tableaux of shape } \bm{\lambda}   \}.\\
\sty_d({\bm \lambda}) &= \{\text{standard }d\text{-tuples of tableaux of shape } \bm{\lambda}   \}.
\end{align*}

If $\bm{T}=(T^{(0)}, \ldots, T^{(d-1)}) \in \ssty_d(\bm{\lambda})$ has entries $1^{\pi_1},2^{\pi_2}, \ldots$, then we say that $\bm{T}$ has shape $\bm{\lambda}$ and weight $\pi=(\pi_1,\pi_2,\ldots)$. For a $d$-tuple of skew shapes $(\lambda^{(0)},\lambda^{(1)}, \ldots, \lambda^{(d-1)})$, define the shifted content of a cell $x$ by
$$\tilde{c}(x)= d \cdot c(x)+ i$$
where $x$ is a cell of $\lambda^{(i)}$, where $c(x)$ is the usual content of $x$ regarded as a cell of $\lambda^{(i)}$. For $\bm{T} \in \ssty_d(\bm{\lambda})$, let $\bmt(x)$ be the entry of the cell $x$ in $\bmt$. Define the set of inversions of $\bmt$ by
$$\Inv_d(\bmt)= \{(x,y) \mid d> \tilde{c}(y)- \tilde{c}(x)>0 \text{ and } \bmt(x)>\bmt(y) \},$$
and the inversion number of $\bmt$ is given by $|\Inv_d(\bmt)|$, denoted by $\inv_d(\bmt)$.\\

By \cite{HHL05}, the LLT polynomial 
$\tilde{G}_{\bmlambda}[X;q]$ is given by
$$\tilde{G}_{\bmlambda}[X;q]=\sum_{\bmt \in \ssty_d(\bmlambda)} q^{\inv_d(\bmt)}x^{\bmt},$$
where $x^{\bmt}$ is the monomial $x_1^{\pi_1}x_2^{\pi_2}\cdots$ when $\bmt$ has weight $\pi$.\\

Define the \emph{content reading word} of a $d$-tuple of tableaux to be the word obtained by reading entries in increasing order of shifted content. For a word $\bm{v}=v_1v_2 \ldots v_n$, define the descent set $\des(v) := \{ i \mid v_i>v_{i+1} \}$ of $v$, and define $\des(\bmt)$ by the descent set of the content reading word of $\bmt$. Then one can write LLT polynomial in terms of Gessel's fundamental quasisymmetric functions:
$$\tilde{G}_{\bmlambda}[X;q]= \sum_{\bmt \in \sty_d(\bmlambda)} q^{\inv_d(\bmt)} F_{\des(\bmt)}.$$

It is known by Grojnowski and Haiman \cite{GH07} that LLT polynomials are Schur-positive. Moreover, Haiman and Haglund conjectured the following stronger statement.

\begin{conj}\label{haiman}
Let $\omega$ be the conjugate map on the symmetric function ring. Then for an integer $p$ and $k>0$ such that for any cell $x$ in $\bmlambda$ the content $c(x)$ satisfies $p\leq c(x) <p+k$ then $\omega \tilde{G}_{\bmlambda}[X;q]$, the conjugate of the LLT polynomial for $\bmlambda$, is $k$-Schur positive.
\end{conj}

\begin{rmk} \label{1} The easiest case is $k=1$, where LLT polynomials are determined by the number $n$ of cells in $\bmlambda$. In this case, $\omega \tilde{G}_{\bmlambda}[X;q]$ is equal to the Hall-Littlewood polynomial $H_{1^n}[X,q]$ which is also a $1$-Schur function $\kso_{1^n}$. 
\end{rmk}
There are two different ways of defining LLT polynomials that Grojnowski and Haiman \cite{GH07} showed that they are essentially the same. In \cite{GH07}, the LLT polynomial defined above is called the new variant combinatorial LLT polynomial. Another one is called the combinatorial LLT polynomial $G'_{\mu^{(1)}/\mu^{(2)}}[X;u]$ defined in terms of the ribbon tableau generating function in a skew shape $\mu^{(1)}/\mu^{(2)}$ such that the power of $u$ encodes the spin statistics of the ribbon tableau, defined by Lascoux, Leclerc and Thibon \cite{LLT}. Grojnowski and Haiman showed in \cite[Proposition 6.17]{GH07} that a new variant combinatorial LLT polynomial for $\bmlambda$ is the same as a certain power of $u$ times the combinatorial LLT polynomial $G'_{\mu^{(1)}/\mu^{(2)}}[X;u^{-1}]$ for certain skew shape $\mu^{(1)}/\mu^{(2)}$ where $u^2=q$. Haiman announced the conjecture in ICM 2006 talk that the combinatorial LLT polynomials are $k$-Schur positive, and this conjecture is equivalent to Conjecture \ref{haiman}. It seems that Haiman and Haglund contributed together to build the conjecture around 2004.\\

\section{Unicellular LLT polynomials}
In this section, we consider unicellular LLT polynomials, i.e. $\lambda^{(i)}$ consists of one box with content $c_i$ for some $c_i \in \mathbb{Z}$ for all $0\leq i \leq d-1$. For the rest of this section, we assume that $\bmlambda$ satisfies the above condition unless stated otherwise. For such a $\bmlambda$, one can associate a permutation $w_{\bmlambda}$ in $S_d$ defined by
$$w_{\bmlambda}(i+1)= |\{j \mid \tilde{c}_j<\tilde{c}_i  \}| +1,$$
for $0\leq i \leq d-1$ where $\tilde{c}_i=d\cdot c_i+i$.

Let $n$ denote the number of boxes in $\bmlambda$, and in this case $n$ is equal to $d$.
For a positive integer $i \leq n$, define the number $f(i)$ by the cardinality of the set 
$$\{j \mid \tilde{c}_j<\tilde{c}_{w_\bmlambda(i)}-d\}.$$
Note that we have $0\leq f(i) \leq i-1$. Moreover, if $i<j$ then $f(i)\leq f(j)$ by definition.\\

For $\bmlambda$, one can associate a partition $\lambda$ contained in a staircase shape $(n-1,n-2,\ldots,1)$ defined by $\lambda_i=f(n-i+1)$. One can compute the LLT polynomial $\tilde{G}_{\bmlambda}[X;q]$ in terms of $\lambda$.\\

 Let $D_\lambda$ be the set of cells in $\lambda$. For a permutation $v$ in $S_n$ and a set $D$ of some boxes satisfying $1\leq i < j \leq n$, define $inv(v,D)$ by the number of inversions $(p,q)$ of $v$ such that $(p,q)\in D$. Then the following holds by definition of $D_\lambda$.

\begin{proposition}\label{D} $\tilde{G}_{\bmlambda}[X;q]$ is the same as
\begin{align}\label{Dlambda}\sum_{v\in S_n} q^{inv(v,D_\lambda)} F_{D(v^{-1})}.
\end{align}
\end{proposition}
For the rest of the paper, we denote the inverse descent set $D(v^{-1})$ by $\id(v)$. If $\bmlambda$ is a tuple of single cells, we denote $\tilde{G}_{\bmlambda}[X;q]$ by $G_{\lambda}[X;q]$.
\begin{rmk}
All cells of $\bmlambda$ lie in two adjacent diagonals if there exists a positive integer $m\leq n$ such that $\lambda$ is contained in a rectangle $m \times (n-m)$. Conversely, if $\lambda$ is contained in $m\times (n-m)$ for some $m$, then there is $d$-tuple of cells $\bmmu$ lying in two adjacent diagonals such that the corresponding partition $\mu$ is equal to $\lambda$.
\end{rmk}
Let $L(n,\lambda;q)$ denote the conjugate of $G_{\bmlambda}[X;q]$. One can show that $q^p \cdot L(n,\lambda;q^{-1})=G_{\bmlambda}[X;q]$, where $p$ is $\frac{n(n-1)}{2} - |\lambda|$ from \cite[Lemma 6.12]{GH07} and \cite[Proposition 6.17]{GH07}. We denote $L(n,\lambda;q)$ by $L(n,\lambda)$ in short. \\

One can generalize the formula (\ref{Dlambda}) in the following way. For a permutation $w$ in $S_n$ and a $n \times n$ matrix $M=(m_{ij})$ with $m_{ij}=0$ unless $i<j$, define an inversion number of $w$ with respect to $M$ defined by $$\sum_{1\leq i <j \leq n} m_{ij}\cdot \chi\left( (i,j) \in \Inv(w)\right)$$
and we denote by $\inv(w,M)$. It is clear that $\inv(w,M)$ is the usual inversion number if $m_{ij}=1$ for all $i<j$, and it is the same as the major if $m_{i,i+1}=i$ for all $i$ and 0 otherwise. For the rest of this paper, the matrix $M$ satisfies the following condition: if $m_{ij}=0$ for some $i<j$, then $m_{kl}=0$ for any $k<i<j<l$. We call this condition $(*)$.\

\begin{definition} For such a matrix $M$, define a function $G_M[X;q]$ defined by
$$\sum_{w \in S_n} q^{\inv(w,M)} F_{\id(w)}.$$
\end{definition}

A relationship between $G_M[X;q]$ and unicellular LLT polynomials is as follows. If all entries of $M$ are either 1 or 0 and satisfies $(*)$, then a function $G_M[X;q]$ is the unicellular LLT polynomial $G_{\lambda}[X;q]$ by Theorem \ref{D}, where $\lambda_i$ is the maximum of $\text{max}\{ j \mid m_{il}=0 \text{ for all } l \leq j\}$ and $n+1-i$.\\

The following theorem provides a local linear relation between unicellular LLT polynomials.

\begin{thm}\label{1convthm}
 For a partition $\lambda$ and $i$ such that $\lambda_i +2 \leq \lambda_{i-1}$. If $i=1$, then let $\lambda_0$ be infinity. Let $\mu^0=\lambda, \mu^1,\mu^2$ be partitions defined by $\mu^a_j= \lambda_j $ if $j\neq i$ and $\mu^a_i=\lambda_i +a$ for $a=0,1,2$. Then
$$G_{\mu^0}[X;q] - G_{\mu^1}[X;q] = q (G_{\mu^1}[X;q] - G_{\mu^2}[X;q] ).$$
when $\lambda_{n-\lambda_i-1}=\lambda_{n-\lambda_i }$. By taking the conjugation, we have
$$L(n,\mu^0) - L(n,\mu^1)= q (L(n,\mu^1)- L(n,\mu^2) ).$$
\end{thm}
\begin{rmk}

The last condition $\lambda_{n-\lambda_i-1}=\lambda_{n-\lambda_i }$ holds if $\lambda,\mu^1,\mu^2$ are contained in a rectangle $ m \times (n-m)$ for some $m$. Indeed, since $\lambda_i+2 \leq n-m$ we have $n-\lambda_i-1 \geq m+1$ so that $\lambda_{n-\lambda_i-1}=\lambda_{n-\lambda_i}=0$.
\end{rmk}
\begin{rmk}\label{g1g2}
Theorem \ref{1convthm} is equivalent to the statement that there are two symmetric functions $g_1,g_2$ satisfying
\begin{align*}
G_{\mu^0}[X;q]&=g_1+g_2\\
G_{\mu^1}[X;q]&=g_1+ q^{-1}g_2\\
G_{\mu^2}[X;q]&=g_1+q^{-2} g_2.
\end{align*}
In fact, we will show that if $\mu^a$ are contained in the rectangle $m \times (n-m)$ then the conjugate of $g_1$ and $g_2$ are both 2-Schur positive. 

\end{rmk}
{\it Proof of Theorem \ref{1convthm}.} 
We use a bijection from $S_n$ to $S_n$ described in \cite{Kad85}. \\

For $1 \leq x<y \leq n$, let $A_{x,y}$ be the subset $\{ w \in S_n \mid (x,y) \in \Inv (w)\}$ of $S_n$. Consider the following bijection:
$$f_{xy} : A_{x,y} \longrightarrow A_{x+1,y} $$
for $x+1<y$, defined by $f_{xy}(w)=w$ if both $(x,y),(x+1,y)$ are in $\Inv (w)$, and $f_{xy}(w)= ws_x$ if $(x+1,y)$ is not in $\Inv(w)$. In the second case, $(x,y)$ is not in $\Inv ( ws_x)$ and $(x+1,y)$ is in $\Inv ( ws_x)$ by the construction. A. Similarly, a slight modification of the map $f_{xy}$ and the proof of the bijectivity of $f_{xy}$ provide a bijection $f'_{xy} $ between $A^c_{x,y}$ and $A^c_{x+1,y}$, defined by $f'_{xy}(w)=w$ if $(x+1,y) \notin \Inv(w)$ and $f'_{xy}(w)= w s_x$ if $ (x+1,y) \in \Inv(w)$. We denote this bijection by the same notation $f_{xy}$ so that $f_{xy}$ is a bijection from $S_n$ to $S_n$.\\

To prove the theorem, set $x= \lambda_i +1=\mu^1_i$ and $y=n+1-i$.
It turns out that the above bijection preserves $\inv(w,D_{\mu^1})$. First of all, it is obvious when $f_{xy}$ is the identity map. If $f_{xy}(w)= ws_x$, for $M=M_{\mu^a}$ ($a=0,1,2$), we have $M_{x,r}=M_{x+1,r}$ for $r>x+1$ not equal to $n+1-i$ by the construction of $\mu^a$. Moreover, $M_{s,x}=M_{s+1,x}$ for all $s+1<x$ by the condition $\lambda_{n-\lambda_i-1}=\lambda_{n-\lambda_i }$ for $M=M_{\mu^a}$. If $\lambda$ is contained in a rectangle $m\times (n-m)$, these numbers are 0.\\

Now we are ready to prove Theorem \ref{1convthm}. Let $g_1,g_2$ be the symmetric functions
\begin{align*}
g_1&=\omega\left( \sum_{w \in A^c_{xy}} q^{\inv(w,M_\lambda)}F_{\id(w)}\right)\\
g_2&=\omega\left( \sum_{w \in A_{xy}} q^{\inv(w,M_\lambda)}F_{\id(w)}\right),
\end{align*}
where $\omega$ is the conjugation. It is clear from the definition that $G_\lambda=g_1+g_2$ and $G_{\mu^1} = g_1 + q^{-1} g_2$. Therefore, it is enough to show that $G_{\mu^2}=g_1 + q^{-2} g_2$. By the bijection $f_{xy}$, we have

\begin{align*}
g_1&=\omega\left( \sum_{w \in A^c_{x+1,y}} q^{\inv(w,M_{\mu^1})}F_{\id(w)}\right)\\
g_2&=\omega\left( \sum_{w \in A_{x+1, y}} q^{\inv(w,M_{\mu^1})-1}F_{\id(w)}\right),
\end{align*}
Since $w \in A_{x+1,y}$, we have $(x+1,y) \in \Inv(w)$, $\inv(w,M_{\mu^1})=\inv(w,M_{\mu^2})+1$ and
\begin{align*}
g_1&=\omega\left( \sum_{w \in A^c_{x+1,y}} q^{\inv(w,M_{\mu^2})}F_{\id(w)}\right)\\
g_2&=\omega\left( \sum_{w \in A_{x+1, y}} q^{\inv(w,M_{\mu^2})-2}F_{\id(w)}\right),
\end{align*}
Therefore, we have
\begin{align*} G_{\mu^2}[X;q]&= \omega\left( \sum_{w \in A^c_{x+1,y}} q^{\inv(w,M_{\mu^2})}F_{\id(w)}\right)+\omega\left( \sum_{w \in A_{x+1, y}} q^{\inv(w,M_{\mu^2})-2}F_{\id(w)}\right) \\
&= g_1+q^{-2}g_2,\\
\end{align*}
we are done.
\qed
\begin{example}
For a partition $\mu$ with $\mu_1\leq 2$, let $\kst_\mu$ denote the $2$-Schur function indexed by $\mu$.
Let $n=6$, $\lambda=(1,1)$, and $i=1$. Then we have $\mu^1=(2,1)$, $\mu^2=(3,1)$, and

\begin{align*}
L(n,\lambda)&= \kst_{1,1,1,1,1,1}+ (q^4+2q^3)\kst_{2,1,1,1,1}+(2q^6+q^5)\kst_{2,2,1,1}+ q^7 \kst_{2,2,2}\\
L(n,\mu^1)&= \kst_{1,1,1,1,1,1}+ 3q^3\kst_{2,1,1,1,1}+3q^5\kst_{2,2,1,1}+ q^6 \kst_{2,2,2}\\
L(n,\mu^2)&= \kst_{1,1,1,1,1,1}+ (2q^3+q^2)\kst_{2,1,1,1,1}+(q^5+2q^4)\kst_{2,2,1,1}+ q^5 \kst_{2,2,2}
\end{align*}
Therefore, one can take
\begin{align*}
g_1&=\kst_{1,1,1,1,1,1}+ 2q^3\kst_{2,1,1,1,1}+q^5\kst_{2,2,1,1}\\
g_2&=q^4\kst_{2,1,1,1,1}+ 2q^6\kst_{2,2,1,1}+ q^7 \kst_{2,2,2}
\end{align*} 
so that Remark \ref{g1g2} holds.
\end{example}





\section{linearlity theorem}
In this section, we show Theorem \ref{convthm}. 
\begin{thm}[Linearity theorem] \label{convthm}
Let $\lambda$ be the partition contained in $(n-m)^{m}$. For a subset $I$ of $\{1,2,\ldots,m\}$ and $i\in I$, define $e_i$ be $1$ if $i\in I$ and $0$ otherwise. Then there are $2$-Schur positive functions $f_{I,m}$ such that
$$L(n,\lambda)= \sum_{I \subset [m]} f_{I,m} q^{- e_I\cdot \lambda} $$
where $e_I\cdot \lambda = \sum_{j=1}^m e_i \lambda_i$. 
\end{thm}

\begin{rmk}\label{1imply}
Note that Theorem \ref{convthm} implies Theorem \ref{1convthm} for any choice of $\lambda,\mu^1,\mu^2$ satisfying the hypothesis in Theorem \ref{1convthm}. Conversely, assuming the hypothesis in \ref{1convthm} if we know that Theorem \ref{convthm} is true for two of $\lambda,\mu^1,\mu^2$, then Theorem \ref{convthm} is true for all three partitions.
\end{rmk}

Without loss of generality, one can assume that $m\leq n/2$ since one can show that $L(n,\lambda)$ is the same as $L(n,\lambda')$ where $\lambda'$ is the conjugate of $\lambda$. The exact formula for $f_{I,m}$ when $m=[n/2]$ is the following.

\begin{thm}\label{half}
Assume that $m=[n/2]$. 
Let $l$ be the size of $I$. If $I=\{m-l+1,m-l+2, \ldots, m \}$, then $f_{I,m}= q^{ml} \kst_{2^l 1^{n-2l}}$. For a subset $I$, there is a partition $\mu$ such that
$$I= \{ m-l+1-\mu_1, m-l+2-\mu_2, \ldots, m- \mu_l \}$$
Then $f_{l,m}= q^{ml+|\mu|} \kst_{2^l 1^{n-2l}}$ if $n$ is even, and $f_{l,m}= q^{(m+1)l+|\mu|} \kst_{2^l 1^{n-2l}}$ if $n$ is odd.
\end{thm}
Theorem \ref{half} follows from Theorem \ref{2schur} and Theorem \ref{fim}. Note that for fixed $l$, the maximum of $ml+|\mu|$ for even $n$ and the maximum of $(m+1)l+ |\mu|$ for odd $n$ are the same as $l(n-l)$.

\begin{rmk}
Note that the number of weak standard tableau of length $n$ for $k=2$ is $2^{[n/2]}$ \cite{LLMS10} which is the same as the number of possible $I$ for $m=[n/2]$. 
\end{rmk}

We will first show Theorems \ref{convthm}, \ref{half} for $m=[n/2]$ when $n$ is even. Theorem \ref{convthm} for $m<[n/2]$ can be proved similarly. Formulas for $f_{I,m}$ when $m\leq[n/2]$ is the following.

\begin{proposition}\label{less} For $m\leq [n/2]$, we have
$$f_{J,m}= \sum f_{I,[n/2]}$$
where the sum runs over all subset $I \subset \{1,2,\ldots,[n/2]\}$ such that $J=I \cap \{1,2,\ldots,m\}$.
\end{proposition}

Note that the formula in Proposition \ref{less} is consistent with Theorem \ref{convthm} when fixing $\lambda$ and varying the integer $m$.\\

For proving Theorem \ref{convthm} for $n=[m/2]$, we need a relation between LLT polynomials for dominos, $2$-Schur functions, and the function $f_{I,m}$. It turms out that $f_{I,m}$ is a power of $q$ times single $2$-Schur function by following two theorems.

\begin{thm}\label{2schur}
Let $\eta_0$ be the horizontal domino $(2)$ and $\eta_1$ be the vertical domino $(1,1)$. 
For a $0-1$ sequence $a=(a_1,\ldots,a_m)$ and $n=2m$, define the $m$-tuple of tableaux $\bmeta_a=(\eta_{a_1},\ldots,\eta_{a_m})$ such that all tableaux lie in two adjacent diagonals with contents $0$ and $1$. Then we have
$$\omega(\tilde{G}_{\bmeta_a})=q^{M} \kst_{2^l 1^{n-2l}}$$
where $l$ is the number of $1$ in $a$ and $M=\sum_{i=1}^{\ell} (m-i)$. Similarly, for $n=2m+1$ define $\bmeta_a$ by $(\eta_{a_1},\ldots, \eta_{a_m},b)$ where $b$ is a box with a content $0$. Then we have
$$\omega(\tilde{G}_{\bmeta_a})=q^{M'} \kst_{2^l 1^{n-2l}}$$
where $l$ is the number of $1$ in $a$ and $M'=\sum_{i=1}^{\ell} (m-i+1)$.
\end{thm}
The proof of Theorem \ref{2schur} will be given in the next section. We first prove Theorem \ref{convthm} for the case $\delta_{m-1} \subset \lambda \subset \delta_m$ by writing $f_{I,m}$ in terms of $\omega(\tilde{G}_{\bmeta_a})$.

\begin{thm}\label{fim}
Let $\delta_{m}$ be the staircase tableau $(m,m-1, \ldots, 1)$ and $m=[n/2]$. Then Theorem $\ref{convthm}$ is true for $\delta_{m-1} \subset \lambda \subset \delta_m$ with
$$f_{I,m}= q^{e_I \delta_{m}} \omega(\tilde{G}_{\bmeta_{a_I}})$$
where $a_i$ is 1 if $i \in I$ and $0$ otherwise.
\end{thm}
\begin{proof} Let $\lambda$ be $\delta_{m} - b_\lambda$ where $b_\lambda$ is a $0-1$ sequence of length $m$. Then the theorem is equivalent to the formula
\begin{align}\label{linear}G_\lambda[X;q]= \sum_{I \subset \{1,2,\ldots,m\}}q^{e_I\cdot b_\lambda } \tilde{G}_{\bmeta_{a_I}}[X;q].\end{align}
which is clear by the following lemma:

\begin{lem}\label{4.8}
Let $\bmlambda$ be $d$-tuple of skew partitions such that $\lambda^{(i)}$ is a single box with a content $\epsilon$ for $\epsilon=0$ or $1$ and $\lambda^{(i+1)}$ is a single box with a content $1-\epsilon$. Then
$$\tilde{G}_\bmlambda= \tilde{G}_{\bmmu_0}+q^\epsilon \tilde{G}_{\bmmu_1}$$
where $\bmmu_\epsilon$ is a $(d-1)$-tuple of skew partitions obtained by replacing two partitions $\lambda^{(i)}$ and $\lambda^{(i+1)}$ by a domino $\bmeta_\epsilon$.
\end{lem}
The lemma is obvious by the definition of LLT polynomials and the theorem follows.
\end{proof}

Now we are ready to prove Theorem \ref{convthm}.
\begin{lem}\label{deltam}
For $m=[n/2]$, assume that Theorem \ref{convthm} holds for $\lambda$ satisfying $\delta_{m-1} \subset \lambda \subset \delta_m$. Then Theorem \ref{convthm} is true for all $\lambda \subset \delta_m $. In particular, Theorem \ref{convthm} holds for all $\lambda \subset \delta_p$ for any $p\leq m$.
\end{lem}
\begin{proof}
During the proof, we will only use Theorem \ref{1convthm}. Let $p=p_\lambda$ be the minimum in the set $\{ i \mid \lambda_i \neq m+1-i,m-i \}$. If the set is empty, we set $p_\lambda=m$. Note that the minimum of the set can not be $m$ since $\lambda_m$ is either $0$ or $1$. \\

We use an induction on $p_\lambda$. The base case $p_\lambda=m$ follows from Theorem \ref{2schur} and Equation (\ref{linear}). Assume that Theorem \ref{convthm} is true for all $\lambda$ with $r<p_\lambda$ and let $\nu$ be a partition with $r=p_{\nu}$. For an integer $\nu_{r+1} \leq a \leq m+1-r$, let $\mu^a$ be partitions with $\mu^a_i=\nu_i$ for $i \neq r$ and $\mu^a_{r} = a$. Note that Theorem \ref{convthm} is true for $\mu^a$ with $a=m+1-r,m-r$ by the induction hypothesis. By Theorem \ref{1convthm}, if we have $\mu^{m+1-r}= g_1 + g_2$ and $\mu^{m-r} = g_1 + q g_2$, then $\mu^a= g_1 + q^{m+r-1-a} g_2$. By Remark \ref{1imply}, Theorem \ref{convthm} is true for all $\mu^a$, including $\nu$.
\end{proof}
{\it Proof of Theorem \ref{convthm}.} Without lose of generality we assume that $m\leq [n/2]$. Assume that there is $\lambda$ that does not satisfy Theorem \ref{convthm} with $f_{I,m}$ determined by Theorem \ref{fim} when $m=[n/2]$, and Proposition \ref{less} when $m<[n/2]$. Choose $\lambda$ with minimal size. We will show that for every $i$, $\lambda$ satisfies the inequality $\lambda_{i}\leq \lambda_{i+1}+1$. If this is true, then we have $\lambda\subset \delta_m$ and we have a contradiction by Lemma \ref{deltam}.\\

Assume that there is $i$ such that $\lambda_i \geq \lambda_{i+1}+2$. For $a=1,2$, define $\mu^a$ be the partition defined by $\mu^a_j= \lambda_j$ for $j \neq i$ and $\mu^a_i=\lambda_i-a$. Since $|\mu^a|<|\lambda|$, $\mu^a$ for $a=1,2$ satisfy Theorem \ref{convthm}. Therefore by Remark \ref{1imply} $\lambda$ also satisfies Theorem \ref{convthm} which makes a contradiction. \qed

Note that the same argument in the proof of Theorem \ref{convthm} can be used to prove Proposition \ref{less}, proving Theorem \ref{convthm} for $m\leq [n/2]$.



 \section{LLT polynomials for dominos and $2$-Schur functions}\label{dominos}
 In this section, we prove Theorem \ref{2schur}. To prove Theorem \ref{2schur}, we introduce generalized Hall-Littlewood polynomials and $2$-Schur functions. Then we show that $\omega(\tilde{G}_{\bmeta_a})$ appeared in Theorem \ref{fim} is the same as $2$-Schur functions up to a power of $q$.
\subsection{Generalized Hall-Littlewood polynomials}
For a symmetric function $f,g$, let $f^\perp(g)$ be the unique symmetric function satisfying $\langle g, fh \rangle = \langle f^\perp(g) , h\rangle$ for any symmetric function $h$. One can use this to define \emph{creation operators} $\varmathbb{S}_a$ for the Schur functions, defined by $\varmathbb{S}_a = \sum_{r \geq 0} (-1)^r h_{m+r} e^\perp_{r}$ where $h_r$ (resp. $e_r$) is the complete homogeneous symmetric function (resp. the elementary symmetric function) of degree $r$.

 Jing \cite{Jin} defined the operators $\varmathbb{B}_a$ for a positive integer $a$ that generalize the creation operators for Schur functions. Define
 $$\varmathbb{B}_a = \sum_{i,j \geq 0} (-1)^i q^j h_{a+i+j} e^\perp_i h^\perp_j = \sum_{j \geq 0} q^j \varmathbb{S}_{a+j} h_j^\perp.$$
 Let $H_\lambda[X;q]$ be the Hall-Littlewood polynomials. Then this family of operator has the property that
 $$\varmathbb{B}_a(H_\lambda[X;q]) = H_{(a,\lambda_1,\lambda_2,\ldots,\lambda_\ell)}[X;q].$$
Shimozono and Zabrocki generalized the operators for any partition by
$$\varmathbb{B}_\lambda = \det ( \varmathbb{B}_{\lambda_i+j-i} )_{1\leq i,j \leq \ell}.$$
Note that if $q=0$, then $\varmathbb{B}_a=\varmathbb{S}_a$ and $\varmathbb{B}_\lambda(1) = s_\lambda$.\\

For tuple of partitions $\bmlambda = (\lambda^{(0)},\lambda^{(1)}, \ldots, \lambda^{(d-1)})$, consider the \emph{generalized Hall-Littlewood polynomials} defined by 
$$K_{\bmlambda}[X;q]:=\varmathbb{B}_{\lambda^{(0)}}  \cdots \varmathbb{B}_{\lambda^{(d-1)}}(1).$$
Grojnowski and Haiman \cite[Theorem 7.5]{GH07} showed the following.

\begin{thm}\label{gh}
If $\lambda^{(i)}$ is a rectangle for all $i$ such that contents of southwest corners of $\lambda^{(i)}$ are weakly increasing and contents of southeast corners of $\lambda^{(i)}$ are weakly decreasing, $K_{\bmlambda}[X;q]$ is equal to a power of $q$ times $\omega(\tilde{G}_{\bmlambda'})$ where $\bmlambda'=(\lambda^{(0)'},\ldots, \lambda^{(d-1)'})$.
\end{thm}
For our case, we set $\bmlambda=\bmeta_a$ with $a=(0,0,\ldots,0,1,\ldots, 1)$ so $\lambda^{(i)}$ are either a horizontal domino, a vertical domino, or a single box. We show in Subsection 5.3 that $\omega(\tilde{G}_{\bmeta_a})$ only depends on the number $\ell$ of vertical dominos.
\subsection{$k$-Schur functions}
In this subsection, we show that for $a=(1^\ell,0^{m-\ell})$ and $m=[n/2]$, $K_{\bmeta_a}$ is equal to $\kst_{2^\ell 1^{n-\ell}}$.
We first recall the algebraic definition of the $k$-Schur functions \cite{LM01}. A partition $\lambda$ is called $k$-bounded if $\lambda_1\leq k$. For a partition $\lambda$, let $h(\lambda)$ be the main hook-length of $\lambda$ defined by $\lambda_1+ \ell(\lambda)-1$ where $\ell(\lambda)$ is the number of parts in $\lambda$. For a $k$-bounded partition $\lambda$, let $\lambda^{\rightarrow k}$, called the $k$-split of $\lambda$ be $(\lambda^{(1)},\ldots,\lambda^{(r)})$ where its concatenation is equal to $\lambda$, $h(\lambda^{(i)})=k$ for all $i<r$, and $h(\lambda^{(r)})\leq k$. For example, $(3,2,2,2,1,1)^{\rightarrow 3} = \left( (3),(2,2),(2,1),(1) \right)$.\\

For a $k$-bounded partition, let $\lambda^{\rightarrow k} =(\lambda^{(1)},\ldots,\lambda^{(r)})$. The $k$-split polynomials are defined recursively by 
$$C^{(k)}_\lambda[X;q] = \varmathbb{B}_{\lambda^{(1)}} C^{(k)}_{(\lambda^{(2)},\ldots, \lambda^{(r)})}$$
with $C^{(k)}_{()}=1$. Let $T^{(k)}$ be an operator acting on the ring of symmetric functions defined by
$$ T^{(k)}_i (C^{(k)}_\lambda[X;q] ) = \begin{cases} C^{(k)}_\lambda[X;q] &\text{if } \lambda_1=i, \\ 0 & \text{otherwise.} \end{cases}$$

Now we are ready to define $k$-Schur functions recursively.
\begin{definition}\label{rectangle}
For a $k$-bounded partition $\lambda$, if $\ell(\lambda)=1$ and $a\leq k$, then $s^{(k)}_{a}$ is the Schur function $s_{a}$. Otherwise for $\lambda_1 \leq m \leq k$, we define $k$-Schur functions by the following recursion:
$$ s^{(k)}_{(m,\lambda)}=T^{(k)}_m \varmathbb{B}_m s^{(k)}_\lambda[X;q].$$
\end{definition}

For $k=2$, the following $k$-rectangular property \cite[Theorem 26]{LM01} of $k$-Schur functions are useful.
\begin{thm}\label{krec} If $\mu,\nu,\lambda$ are partitions where $\lambda=(\mu,\nu)$ and $\mu_{\ell(\mu)}>\ell \geq \nu_1$, then
$$\varmathbb{B}_{\ell^{k+1-\ell}} s^{(k)}_\lambda[X;q] = q^{|\mu|- \ell(\mu)} s^{(k)}_{(\ell^{k+1-\ell}) \cup \lambda}[X;q].$$
In particular, if $\ell\geq \lambda_1$ then 
$$\varmathbb{B}_{\ell^{k+1-\ell}} s^{(k)}_\lambda[X;q] = s^{(k)}_{(\ell^{k+1-\ell}) \cup \lambda}[X;q].$$
\end{thm}
Theorem \ref{krec} implies that $2$-Schur functions are the same as generalized Hall-Littlewood polynomials.

\begin{thm}\label{g2schur} Let $\lambda=(2^\ell, 1^{n-2\ell})$. Then a $2$-Schur function $\kst_{\lambda}[X;q]$ is the same as the generalized Hall-Littlewood polynomial $K_{\lambda^{\rightarrow 2}}[X;q]$ indexed by $2$-split of $\lambda$.

\end{thm}

By Theorem \ref{g2schur} and Theorem \ref{gh}, we know that $\omega(\tilde{G}_{\bmeta_a})$ for $a=(1^\ell,0^{m-\ell})$ is equal to a power of $q$ times $\kst_{2^\ell,1^{n-\ell}}$. To prove Theorem \ref{2schur}, we show that $\omega(\tilde{G}_{\bmeta_a})$ only depends on the number $\ell$ of vertical dominos, and calculate the exponent of $q$.

\subsection{Comparing different LLT polynomials}
 For a skew shape $\lambda$, let the content set of $\lambda$ be the set of contents of all boxes in $\lambda$. 
To prove that $\omega(\tilde{G}_{\bmeta_a})$ only depends on $\ell$, we show the following theorem.

\begin{lem}\label{commute} Let $\bmlambda$ be $d$-tuple of skew partitions such that $\lambda^{(i)}$ is a horizontal domino with a content set $\{0,1\}$ and $\lambda^{(i+1)}$ is a vertical domino with the content set $\{0,1\}$. Then 
$$\tilde{G}_\bmlambda=\tilde{G}_\bmmu$$
where $\bmmu$ is obtained from $\bmlambda$ by swapping $\lambda^{(i)}$ and $\lambda^{(i+1)}$.

\end{lem}
\begin{proof}

We construct a bijection $\Psi$ between $\sty(\bmlambda)$ and $\sty(\bmmu)$ preserving the inversion number and entries in $\lambda^{(j)}$ for $j\neq i,i+1$ to prove the theorem. The bijection will also preserve the set of entries at the diagonal with the content $0$ so that we can assume $i=0$ and $d=2$. Indeed, we are only changing entries in $i$-th and $(i+1)$-th partitions of $\bm{T}=(T^{(0)},\ldots, T^{(d-1)})$ so that the inversion set $\Inv_d(\bm{T})$ does not change except paris $(x,y)$ satisfying $x,y \subset T^{(i)} \cup T^{(i+1)}$.\\

Assume that $i=0$ and $d=2$. For $\bm{T} \in \sty(\bmlambda)$, let 
$$T^{(0)}= \begin{ytableau} a\\b\end{ytableau},\quad T^{(1)}= \begin{ytableau} c&d \end{ytableau}.$$
There are six possible cases for an order of $a,b,c,d$ and we define $\Psi$ case by case.\\
\begin{enumerate}
\item $a>d>c>b$: then the number of inversion is $2$ and set 
$$\Psi(\bm{T})^{(0)}= \begin{ytableau} c & d \end{ytableau} , \quad \Psi(\bm{T})^{(1)}= \begin{ytableau} a \\ b \end{ytableau}.$$
In this case, two pairs $(a,d)$ and $(d,b)$ are inversion pairs.
\item $a>d>b>c$: then the number of inversion is $1$ and set 
$$\Psi(\bm{T})^{(0)}= \begin{ytableau} c & b \end{ytableau} , \quad \Psi(\bm{T})^{(1)}= \begin{ytableau} a \\ d \end{ytableau}.$$
In this case, a pair $(a,b)$ is an inversion pair.
\item $d>a>c>b$: then the number of inversion is $2$ and set 
$$\Psi(\bm{T})^{(0)}= \begin{ytableau} a & d \end{ytableau} , \quad \Psi(\bm{T})^{(1)}= \begin{ytableau} c \\ b \end{ytableau}.$$
In this case, two pairs $(a,c)$ and $(d,b)$ are inversion pairs.
\item $d>a>b>c$: then the number of inversion is $1$ and set 
$$\Psi(\bm{T})^{(0)}= \begin{ytableau} c & d \end{ytableau} , \quad \Psi(\bm{T})^{(1)}= \begin{ytableau} a \\ b \end{ytableau}.$$
In this case, a pair $(d,b)$ is an inversion pair.
\item $a>b>d>c$: then the number of inversion is $2$ and set 
$$\Psi(\bm{T})^{(0)}= \begin{ytableau} c & b \end{ytableau} , \quad \Psi(\bm{T})^{(1)}= \begin{ytableau} a \\ d \end{ytableau}.$$
In this case, pairs $(a,b)$ and $(b,d)$ are inversion pairs.
\item $d>c>a>b$: then the number of inversion is $1$ and set 
$$\Psi(\bm{T})^{(0)}= \begin{ytableau} a & d \end{ytableau} , \quad \Psi(\bm{T})^{(1)}= \begin{ytableau} c \\ b \end{ytableau}.$$
In this case, a pair $(d,b)$ is an inversion pair.

\end{enumerate}
\end{proof}
\subsection{Proof of Theorem 4.6}
Now we are ready to prove Theorem \ref{2schur}. First note that by Theorem \ref{gh}, Theorem \ref{g2schur}, and Lemma \ref{commute}, we know that $\omega(\tilde{G}_{\bmeta_a})$ is equal to $q^{p} \kst_{2^\ell 1^{n-2\ell}}$ for some integer $p$ where $\ell$ is the number of $1$ in $a$. To prove Theorem \ref{2schur} we only need to show that $p=M$ for even $n$ and $p=M'$ for odd $n$. Note that $p$ is the minimum exponent of $q$ appearing in $\omega(\tilde{G}_{\bmeta_a})$, since a $2$-Schur function $\kst_{2^\ell 1^{n-2\ell}}$ has a term $s_{2^\ell 1^{n-2\ell}}$ with the minimum exponent $0$ in its Schur expansion. We show $p=M$ for even $n$, and left to readers for odd $n$.\\

By Lemma \ref{commute}, one can assume that $a=(0^{m-\ell}1^\ell)$. First we show that $p\geq M=\sum_{i=1}^\ell (m-i)$. For $\bm{T}\in \sty(\bmeta_a)$, $m-\ell+1\leq i\leq m$, and $j<i$, consider possible inversion pairs $(x,y)$ where $x$ is in $T^{(j)}$ and $y$ is in $T^{(i)}$. Since $T^{(i)}$ is a vertical domino, there must be at least one inversion pair $(x,y)$ where the cell $x$ is in $T^{(j)}$ with a content $1$ and $y$ is a cell in $T^{(i)}$. Note that the number of $(i,j)$ satisfying $m-\ell+1\leq i\leq m$, and $j<i$ is exactly $M$. Then it is enough to find $\bm{T}\in \sty(\bmeta_a)$, so that there is no other inversion pairs. Indeed, one can choose $\bm{T}$ so that for $0\leq \alpha\leq m-\ell-1$, numbers $\alpha+1$ and $m+\ell+\alpha+1$ appear in $T^{(\alpha)}$ and for $m-\ell \leq \alpha \leq m-1$, numbers $\alpha+1$ and $\alpha+\ell+1$ appear in $T^{(\alpha)}$. Then the theorem follows. \qed



\section{Haiman-Hugland conjecture for $k=2$}
In this section, we provide the $2$-Schur expansion of the LLT polynomial indexed by skew shapes lying in two adjacent diagonals, confirming Conjecture \ref{haiman} for $k=2$. Recall that for a skew shape $\lambda$, let the content set of $\lambda$ be the set of contents of all boxes in $\lambda$. \\

Let $\bmlambda=(\lambda^{(0)},\lambda^{(1)}, \ldots, \lambda^{(d-1)})$ be the $d$-tuple of skew shapes such that $\lambda_{(i)}$ is either a vertical domino with content set $\{0,1\}$, a horizontal domino with content set $\{0,1\}$, a single cell with content $0$, or a single cell with content $1$. Let $n$ be the number of boxes in $\bmlambda$, and let $m$ be the number of boxes in $\bmlambda$ with content $1$. Without loss of generality, one can assume that $m\leq [n/2]$ because otherwise one can replace $\bmlambda$ by the \emph{conjugate} of $\bmlambda$, defined by $\bmmu$ with $\mu^{(i)}=g(\lambda^{(d-1-i)})$ where the map $g(\mu)$ is the identity if $\mu$ is one of dominos and $g(\mu)$ is a cell with content $1-i$ if $i$ is the content of a single cell $\mu$. One can show that LLT polynomials does not change.\\

For $1 \leq i \leq n$, let $x_i$ be the cell in $\bmlambda$ such that the shifted content $\tilde{c}_i$ of the cell $x_i$ satisfies $\tilde{c}_1<\tilde{c}_2<\cdots<\tilde{c}_n$. For a positive integer $i\leq n$, $f'(i)$ is the cardinality of the set 
$$\{j \mid \tilde{c}_j \leq \tilde{c}_i-d\}.$$
Note that we have $0\leq f'(i) \leq i-1$ and $f'(i)$ is weakly increasing as $i$ increases.\\

Let $\lambda$ be the partition defined by $\lambda_i=f'(n-i+1)$. Then $\lambda$ is contained in a rectangle $(n-m)^m$. Recall Theorem \ref{convthm} that we have
$$L(n,\lambda)= \sum_{I \subset [m]} f_{I,m} q^{- e\cdot \lambda} $$
where $e\cdot \lambda = \sum_{j=1}^m e_i \lambda_i$ and $2$-Schur positive functions $f_{I,m}$. For given $\bmlambda$, define a subset $K \subset \{ 1,2,\ldots,m \}$ satisfying $i \in K$ if and only if $x_{n+1-i}$ is contained in either a vertical domino and horizontal domino. Also for $i \in K$ define $\zeta_i$ by 0 if $x_{n+1-i}$ is contained in a horizontal domino, and 1 if $x_{n+1-i}$ is contained in a vertical domino. By Theorem \ref{convthm} and Lemma \ref{4.8}, one can show the following:

\begin{thm}\label{haiman2}For $m\leq [n/2]$, we have
$$G^{(n)}_\bmlambda= \sum_{\substack{I \subset \{1, \ldots, m\} \\ e_i=\zeta_{i} \text{ if } i \in J}} f_{I,m} q^{- e\cdot \lambda-z}$$
where $f_{I,m}$'s are determined by Theorem \ref{half} and Proposition \ref{less}, and $z$ is the number of vertical dominos in $\bmlambda$.
\end{thm}
By Theorem \ref{half} and Proposition \ref{less} we showed that $G^{(d)}_\bmlambda$ is $2$-Schur positive, showing Conjecture \ref{haiman} for $k=2$.

\section{Concluding Remarks}
One of direct corollary from Theorem \ref{convthm} is the $2$-Schur expansion of product of $1$-Schur functions. Note that a product of $k$-Schur function and $k'$-Schur function is conjecturally $(k+k')$-Schur positive.
\begin{cor}
For $m\leq n/2$ and a subset $I$ of $\{1,2,\ldots, [n/2]\}$, let $l$ be the size of $I$ and $l_1$ be the size of $I \cap \{1,2,\ldots,m\}$. Also, let $l_2$ be $\sum_{j\in I} j - { m\choose 2}$.
Then
$$\kso_{1^m}\kso_{1^{n-m}}= \sum_{I \subset \{1,2,\ldots, [n/2]\}} q^{-(n-m)l_1+l(n-l)- l_2} \kst_{2^l 1^{n-2l}} $$
\end{cor}
\begin{proof}
By Remark \ref{1}, the product $\kso_{1^m}\kso_{1^{n-m}}$ is equal to $L(n,(n-m)^m)$. Then the corollary directly follows from Theorem \ref{convthm} and Proposition \ref{less}.
\end{proof}

Also note that Theorem \ref{1convthm} does not require the condition $k=2$, so that some of results in this paper can be generalized. Roughly speaking, Theorem \ref{1convthm} says that exponents of $q$ are changing piece-wise linearly, especially for small $k$. However, the identification of LLT polynomials for rectangles and generalized Hall-Littlewood polynomials doesn't seem to help much for $k>2$ when working with linear relations. In fact, understanding LLT polynomials indexed by ribbons are much more helpful to determine many of unicellular LLT polynomials since it is easier to use Lemma 4.8 and its obvious generalization for the case of ribbons instead of dominos. Note that finding a Schur expansion of LLT polynomials indexed by ribbons is a challenging problem as it provides a formula for Schur expansion of Macdonald polynomials as well.



 


\begin{thebibliography}{99} 
\bibitem{Bla} Blasiak, {\it Haglund's conjecture on 3-column Macdonald polynomials}. Math. Z. 283, (2016), 601--628.
\bibitem{CM} Carlsson and Mellit, {\it A proof of the shuffle conjecture}, J. Amer. Math. Soc. 31 (2018), 661-697.
\bibitem{CH} Cho, Huh, {\it On e-positivity and e-unimodality of chromatic quasisymmetric functions}, arXiv:1711.07152.
\bibitem{GH07} Grojnowski and Haiman, {\it Affine Hecke algebras and positivity of LLT and Macdonald polynomials}. preprint.
\bibitem{Kad85} Kevin W.J Kadell, {\it Weighted inversion numbers, restricted growth functions, and standard young tableaux}, J. Comb. Theory Ser. A, 40 (1985), 22-44.

\bibitem{HHL05} Haglund, Haiman, Loehr, {\it A Combinatorial Formula for Macdonald Polynomials} J. Amer. Math. Soc. 18 (2005), no. 3, 735–761.
\bibitem{HP} Harada, Precup, {\it The cohomology of abelian Hessenberg varieties and the Stanley-Stembridge conjecture}, arXiv:1709.06736.
\bibitem{Jin} Jing, {\it Vertex operators and Hall-Littlewood symmetric functions}, Adv. Math. 87(1991) 226–248.
(1991) 226–248.
\bibitem{LLMS10} Lam, Lapointe, Morse, Shimozono, {\it Affine Insertion and Pieri Rules for the Affine Grassmannian}, Memoirs of Amer. Math. Soc., 2010, 82pp
\bibitem{LLT} Lascoux, Leclerc, and Thibon, {\it Ribbon Tableaux, Hall-Littlewood Functions, Quantum Affine Algebras and Unipotent Varieties}, J. Math. Phys. 38 (1997), no. 2, 1041-1068.
\bibitem{LM01} Lapointe, Morse, {\it Schur function identities, their $t$-analogs, and $k$-Schur irreducibility}, Adv. Math. 180(2003), 222-247.

\end{thebibliography}
\end{document}